\documentclass{sn-jnl-noheader}

\usepackage{amsmath,amsfonts,amssymb,amsthm}

\usepackage[capitalize]{cleveref}

\jyear{2022}%

\theoremstyle{thmstyleone}
\newtheorem{theorem}{Theorem}[section]
\newtheorem{lemma}[theorem]{Lemma}

\newtheorem{corollary}[theorem]{Corollary}

\theoremstyle{thmstyletwo}

\theoremstyle{thmstylethree}

\normalbaroutside

\begin{document}

\title[Tournaments and Even Graphs are Equinumerous]{Tournaments and Even Graphs are Equinumerous}

\author[1]{\fnm{Gordon F.} \sur{Royle}}\email{gordon.royle@uwa.edu.au}
\author[1]{\fnm{Cheryl E.} \sur{Praeger}}\email{cheryl.praeger@uwa.edu.au}
\author[1]{\fnm{S. P.} \sur{Glasby}}\email{stephen.glasby@uwa.edu.au}
\author*[2]{\fnm{Saul D.} \sur{Freedman}}\email{sdf8@st-andrews.ac.uk}
\author[1]{\fnm{Alice} \sur{Devillers}}\email{alice.devillers@uwa.edu.au}

\affil[1]{\orgdiv{Centre for the Mathematics of Symmetry and Computation, Department of Mathematics and Statistics}, \orgname{The University of Western Australia}, \orgaddress{\street{35 Stirling Hwy}, \city{Crawley}, \postcode{6009}, \state{WA}, \country{Australia}}}
\affil*[2]{\orgdiv{School of Mathematics and Statistics}, \orgname{University of St Andrews}, \orgaddress{\city{St Andrews}, \postcode{KY16 9SS}, \country{UK}}}


\abstract{A graph is called \emph{odd} if there is an orientation of its edges and an automorphism that reverses the sense of an odd number of its edges, and \emph{even} otherwise. Pontus von Br\"omssen (n\'e Andersson) showed that the existence of such an automorphism is independent of the orientation, and considered the question of counting pairwise non-isomorphic even graphs. Based on computational evidence, he made the rather surprising conjecture that the number of pairwise non-isomorphic \emph{even graphs} on $n$ vertices is equal to the number of pairwise non-isomorphic \emph{tournaments} on $n$ vertices. 
We prove this conjecture using a counting argument with several applications of the Cauchy-Frobenius Theorem.}



\maketitle

\section{Introduction}

In a paper on the asymptotics of random tournaments,  Pontus Andersson \cite[page 252]{MR1662785} introduced the concept of an \emph{even graph} in the following fashion: Given a graph $X$, assign an arbitrary orientation to its edges. Then an automorphism $g \in \mathrm{Aut}(X)$ \emph{reverses the sense} of an edge $e = \{u,v\}$ if $e$ is oriented from $u$ to $v$, but $e^g$ is oriented from $v^g$ to $u^g$. Changing the orientation of the graph might alter the \emph{number} of edges whose sense is reversed by $g$, but Andersson showed that it does not alter the \emph{parity} of this number. He defined a graph to be \emph{odd} if it has an automorphism reversing the sense of an odd number of edges and \emph{even} otherwise. (This is somewhat overloading the adjectives ``even'' and ``odd'', which are already used for several different concepts relating to both permutations and graphs but, to avoid confusion, we shall be explicit when using any of these other meanings.)

Figure~\ref{fig:4vert} shows the pairwise non-isomorphic graphs on four vertices, along with an odd automorphism, i.e., an automorphism reversing the sense of an odd number of edges, if one exists. The four graphs with no odd automorphism listed are the four even graphs on four vertices.

A \emph{tournament} is a directed graph $D$ with arc set $A(D)$ such that for any distinct vertices $\{v,w\}$, either $(v,w) \in A(D)$ or $(w,v) \in A(D)$, but not both. Equivalently, a tournament is an \emph{oriented complete graph}. 
Figure~\ref{fig:tourn4} illustrates the four pairwise non-isomorphic tournaments on four vertices. 

When counting graphs (or tournaments, even graphs etc.) on $n$ vertices, we normally fix a vertex set of size $n$, and distinguish between counting \emph{pairwise distinct} graphs and counting \emph{pairwise non-isomorphic} graphs as above. For brevity the former is normally referred to as counting \emph{labelled} graphs, and the latter as counting \emph{unlabelled} graphs. As we shall see, it is no coincidence that the numbers of unlabelled even graphs and tournaments on four vertices are the same. Pontus von Br\"omssen computed the number of unlabelled even graphs on up to $10$ vertices to enter the sequence into the On-Line Encyclopedia of Integer Sequences \cite{A000568}. He was surprised to see that his numbers coincided perfectly with the leading entries of the sequence A000568 of \cite{A000568}, which counts the number of unlabelled tournaments. In a comment on the sequence, he asked whether the numbers of such even graphs and tournaments actually coincide for all $n$. In this note, we answer this question positively.

\begin{figure}[t]
\centering
\includegraphics{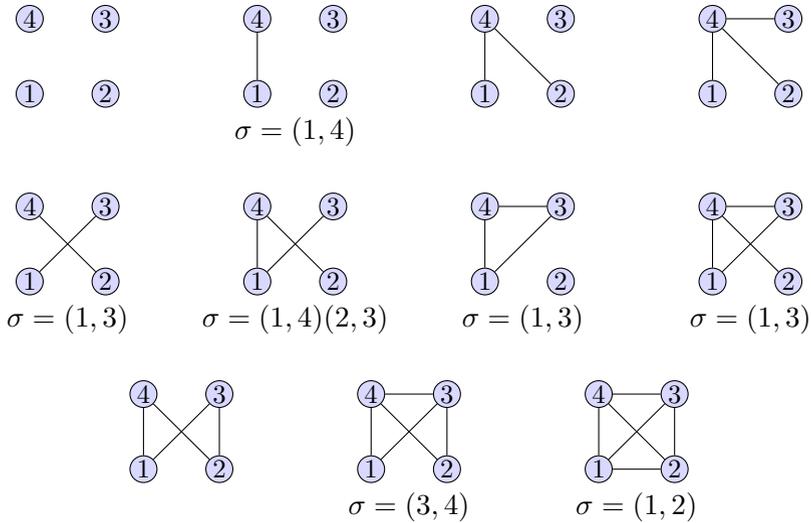}
\caption{$4$-vertex graphs with an odd automorphism if one exists}
\label{fig:4vert}
\end{figure}

\begin{figure}[t]
\centering
\includegraphics{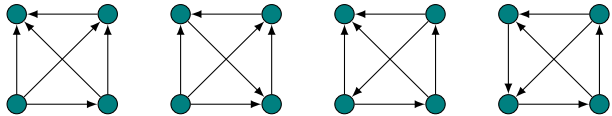}
\caption{Tournaments on four vertices}
\label{fig:tourn4}
\end{figure}

The most obvious and satisfying way to prove that two sequences counting combinatorial structures are the same is to find a reasonably natural bijection between the sets enumerated by the two sequences. This not only proves that the sequences are equal, but also gives a convincing explanation for \emph{why} they are the same. The second-best option is to find a \emph{counting argument} showing that the two sequences are given by the same formula or expression (or satisfy the same recurrence, etc.). By using the Cauchy-Frobenius Theorem and double counting, we find expressions for the numbers of unlabelled graphs, tournaments and odd graphs, showing that for any fixed number of vertices,

\medskip

\begin{center}
Number of graphs = Number of tournaments + Number of odd graphs.
\end{center}

\medskip

As a graph is either even or odd, this immediately implies our main result:

\begin{theorem}
\label{thm:evenandtourn}
For all $n \geqslant 2$, the number of pairwise non-isomorphic even graphs on $n$ vertices is equal to the number of pairwise non-isomorphic tournaments on $n$ vertices.
\end{theorem}

It is an open problem to find a natural bijection between the sets of unlabelled even graphs and tournaments on $n$ vertices.

\section{Counting Graphs}
\label{sec:graphcount}

In general, counting a family of labelled (i.e., pairwise distinct) graphs on $n$ vertices is easier than counting unlabelled (i.e., pairwise non-isomorphic) graphs. The Cauchy-Frobenius Theorem is a fundamental tool that can be used to express the number of unlabelled graphs as a sum, each of whose terms is the size of a specific set of labelled graphs.

Before applying this theorem however, we need to establish some notation. Let $[n] = \{1,2,\ldots,n\}$ denote
the vertex set of all of our graphs and tournaments on $n$ vertices. Then define
\[
E_n = \{ \{u,v\} : u, v \in [n], u \ne v \},
\] 
so that $E_n$ is the set of \emph{unordered pairs} of distinct vertices.
We think of $E_n$ as being the set of \emph{all possible edges} in a graph with vertex set $[n]$.

Suppose that $X$ and $Y$ are graphs with vertex set $[n]$ and  edge-sets $E(X)$ and $E(Y)$ respectively. Then $X$ and $Y$ are \emph{isomorphic} if and only if there is some permutation $g \in \mathrm{Sym}(n)$ such that $E(Y) = E(X)^g$, where
\[
E(X)^g = \{ \{u^g,v^g\} \mid \{u,v\} \in E(X)\}.
\]
So the number of unlabelled graphs is the \emph{number of orbits} of the symmetric group $\mathrm{Sym}(n)$ acting on $\Omega = \mathcal{P}(E_n)$ (the powerset of $E_n$).

In order to count orbits, we turn to the following well-known theorem (see Neumann \cite{MR562002} for the fascinating history of this result).
\begin{theorem}[Cauchy-Frobenius Theorem]\label{thm:cauchyfrob}
Let $G$ be a permutation group acting on a set $\Omega$. The number of orbits of $G$ on $\Omega$ is given by the expression
\[
\frac{1}{|G|} \sum_{g \in G} |\mathrm{fix}(g)|,
\]
where $\mathrm{fix}(g)$ is the set of elements of $\Omega$ fixed by $g$.
\end{theorem}
In other words, the number of orbits of $G$ on $\Omega$ is equal to the \emph{average number of fixed points} of the elements of $G$.

In order to apply this theorem, we need to know how many \emph{subsets of $E_n$} are fixed by a permutation $g \in \mathrm{Sym}(n)$ in its induced action on $E_n$. Equivalently, we need to know how many \emph{labelled graphs} are fixed by $g$.

So for $g \in \mathrm{Sym}(n)$, let $g_E$ denote the permutation induced by $g$ on $E_n$. For example, if $n = 4$ and $g = (1,2,3,4)$, then 
\[
g_E = \left( \{1,2\}, \{2,3\}, \{3,4\}, \{1,4\} \right)\  \left( \{1,3\}, \{2,4\} \right).
\]
Now any subset of $E_4$ that is fixed by $g_E$ must either contain \emph{all} of the pairs $\{ \{1,2\}, \{2,3\}, \{3,4\}, \{1,4\}\}$ or \emph{none} of them, and similarly for $\{\{1,3\}, \{2,4\}\}$. So there are $2^2 = 4$ subsets of $E_n$ (or labelled graphs) fixed by $g_E$.

In general, a subset of $E_n$ is fixed by $g_E$ if and only if it is the union of the cycles of $g_E$ (here, and later, we identify a cycle of $g_E$ with its support, i.e., the set of edges that it moves). So letting $c(g_E)$ denote the number of cycles of $g_E$, there are exactly $2^{c(g_E)}$ labelled graphs fixed by $g_E$. By the Cauchy-Frobenius Theorem, we conclude that the number of isomorphism classes of graphs on $n$ vertices is given by the expression
\begin{equation}\label{exp:graphs}
\frac{1}{n!} \sum_{g \in \mathrm{Sym}(n)} 2^{c(g_E)}.
\end{equation}

\section{Counting Tournaments}

In this section, we will derive an analogous expression for the number of unlabelled tournaments on $n$ vertices, essentially by considering 
\emph{arcs} rather than \emph{edges}. We start by defining 
\[
A_n = \{ (u,v) : u, v \in [n], u \ne v\},
\]
so that where $E_n$ was the set of all possible \emph{edges}, $A_n$ is the set of all possible \emph{arcs}.

So, for $g \in \mathrm{Sym}(n)$, let $g_A$ denote the permutation induced by $g$ on $A_n$. For example, if $n = 4$ and $g=(1,2,3,4)$, then 
{\small
\[
g_A = \left((1,2),(2,3),(3,4),(4,1)\right)   \left((2,1),(3,2),(4,3),(1,4)\right) \left((1,3),(2,4),(3,1),(4,2)\right).
\]}\noindent 
In this case $g_A$ has three cycles, say $C_1$, $C_2$ and $C_3$. The first two cycles $C_1$ and $C_2$ are closely related, in the sense that $C_2$ can be obtained from $C_1$ by reversing every arc, and vice versa. In contrast, reversing every arc in $C_3$ simply yields $C_3$ again.  We use $c(g_A)$ to denote the number of cycles of $g_A$.

We need a little bit more notation and terminology to discuss the general case. 
If $a$ refers to an ordered pair, say $(u,v)$, then $\overline{a}$ refers to its \emph{reverse} $(v,u)$ and $e(a)$ refers to the unordered pair $\{u,v\}$.

Suppose that $C = (a_1, a_2, \ldots, a_k)$ is a cycle of $g_A$ and let
$
\overline{C} = (\overline{a_1}, \overline{a_2}, \ldots, \overline{a_k})
$
be the cycle obtained from $C$ by reversing each arc. If $C = \overline{C}$, then we say that $C$ is \emph{self-paired}, otherwise \emph{non self-paired}.  So in general, $g_A$ has some number (maybe zero) of self-paired cycles, the remaining non-self-paired cycles occur in pairs of the form $\{C, \overline{C}\}$.

In order to use the Cauchy-Frobenius Theorem, we need to understand the relationship between the cycle structures of $g$, $g_A$ and $g_E$. 

If $C$ is a non self-paired cycle $(a_1, a_2, \ldots, a_k)$ of $g_A$, then its \emph{undirected image}
\[
\left( e(a_1), e(a_2), \ldots, e(a_k) \right)
\]
is a cycle of $g_E$, and clearly $\overline{C}$ has the same undirected image. 

If instead $C$ is a self-paired cycle  of $g_A$, then 
\[
C = \left(a_1, a_2, \ldots, a_k, \overline{a_1}, \overline{a_2}, \ldots, \overline{a_k} \right),
\]
for some $a_1,\ldots,a_k \in A_n$. As $e(a) = e(\overline{a})$, if each arc $a$ is replaced by $e(a)$ we obtain the cycle
\[
\left(e(a_1), e(a_2), \ldots, e(a_k), e(a_1), e(a_2), \ldots, e(a_k) \right),
\]
which is a cycle of $g_E$ ``wrapped around twice''. In this case, we define the \emph{undirected image} of $C$ to be the cycle $\left(e(a_1), e(a_2), \ldots, e(a_k)\right)$.

As we shall shortly see, it is the presence or absence of self-paired cycles that is crucial when counting both tournaments and odd graphs. This, in turn, depends on the \emph{order} $|g|$ of $g$.  
\begin{lemma}\label{le:selfpairedeven}
The permutation $g_A$ has a self-paired cycle if and only if $|g|$ is even.
\end{lemma}
\begin{proof}
If $g$ has even order, then it has a cycle of even length, say $(1,2, \ldots, 2k)$. Then the cycle
\[
\left( (1,k+1), (2,k+2), \ldots, (k,2k), (k+1,1), \ldots, (2k,k) \right)
\]
is a self-paired cycle of $g_A$.

Conversely, if $g_A$ has a self-paired cycle, then this cycle has even length, and so $|g_A|$ is even, implying that $|g|$ is even.
\end{proof}

\begin{corollary}\label{cor:tournaut}
The automorphism group of a tournament has odd order. 
\end{corollary}
\begin{proof}
Let $T$ be a tournament, and suppose that $g \in \mathrm{Aut}(T)$. If $g$ has even order, then by \cref{le:selfpairedeven}, $g_A$ has a self-paired cycle $C$, which contains at least one pair of arcs of the form $\{a,\overline{a}\}$.  The tournament is fixed by $g$ if and only if its arc set is a union of cycles of $g_A$, and so $C$ must either be a subset of $A(T)$ or disjoint from $A(T)$. However, neither is possible because exactly one of $a$, $\overline{a}$ is an arc of $T$.
\end{proof}

\begin{theorem}\label{t:tournaments}
The number of isomorphism classes of tournaments on $n$ vertices is given by the expression
\begin{equation}\label{exp:tourn}
\frac{1}{n!} \sum_{\substack{g \in \mathrm{Sym}(n)\\ |g|\ \mathrm{odd}}} 2^{c(g_E)}.
\end{equation}
\end{theorem}

\begin{proof}
We apply the Cauchy-Frobenius Theorem with $G=\mathrm{Sym}(n)$ and $\Omega$ the set of (labelled) tournaments. By \cref{cor:tournaut}, permutations of even order do not contribute to the sum.
If $g$ has odd order, then $g_A$ has no self-paired cycles, and a tournament is fixed by $g$ if and only if its arc set is a union of cycles of $g_A$ containing \emph{exactly one} cycle from each pair $\{C, \overline{C}\}$. Therefore $g$ fixes exactly $2^{c(g_A)/2}$ tournaments. As $g_A$ has no self-paired cycles, $c(g_A)/2 = c(g_E)$ and so the number of isomorphism classes of tournaments on $n$ vertices is given by  \eqref{exp:tourn}.
\end{proof}

An explicit expression for the number of isomorphism classes of tournaments was first given by Davis \cite{MR55294}, although expressed as a sum over conjugacy classes of permutations rather than a sum over individual permutations.

\section{Counting Odd Graphs}

We now derive a similar expression for counting pairwise non-isomorphic \emph{odd graphs} on $n$ vertices. Let $X$ be a graph on the vertex set $[n]$. Whether or not an automorphism $g \in \mathrm{Aut}(X)$ is an odd automorphism for $X$ is independent of the orientation of $X$, and it is convenient henceforth to assume $V(X) = [n]$ and each edge is oriented from the lower- to the higher-numbered vertex.

An \emph{inversion} of a permutation $h \in \mathrm{Sym}(n)$ is a pair $\{u,v\}$ of distinct elements of $[n]$ such that $u<v$ and $u^h > v^h$. An edge of $X$ has its sense reversed by a permutation $g \in \mathrm{Aut}(X)$ if and only if it is an inversion of $g$. So the set of edges of a graph $X$ whose sense is reversed by a permutation $g \in \mathrm{Aut}(X)$ is the set of inversions of $g$ that are also edges of $X$.

Now, given $g \in \mathrm{Aut}(X)$, the \emph{sign} $\mathrm{sgn}_X(g)$ of $g$ with respect to $X$ is the value
\[
\mathrm{sgn}_X(g) = \prod_{\{u,v\} \in E(X)} \frac{u^g - v^g}{u - v}.
\]
Each term in the product is well-defined because $\frac{u^g - v^g}{u - v} = \frac{v^g - u^g}{v - u}$. Because $g \in \mathrm{Aut}(X)$, the numerator and denominator of this expression are each a product of the same multiset of values, but permuted, and possibly with changes of sign. It follows therefore that $\mathrm{sgn}_X(g) \in \{-1,1\}$. 

An edge $\{u,v\}$ contributes $\pm 1$ to the product, with the contribution being $-1$ if and only if its sense is reversed by $g$. Therefore $g$ is an odd automorphism for $X$ if and only if $\mathrm{sgn}_X(g) = -1$.

\begin{lemma}
The function $\mathrm{sgn}_X$ is a group homomorphism from $\mathrm{Aut}(X)$ to the multiplicative group $(\{-1,1\}, \times)$.
\end{lemma}
\begin{proof}
Suppose that $g$, $h \in \mathrm{Aut}(X)$. Then
\begin{align*}
\mathrm{sgn}_X(gh) &= \prod_{\{u,v\} \in E(X)} \frac{u^{gh}-v^{gh}}{u - v}\\
&= \prod_{\{u,v\} \in E(X)} \frac{u^g-v^g}{u - v} \cdot \frac{(u^g)^h - (v^g)^h} {u^g-v^g}\\
&= \prod_{\{u,v\} \in E(X)} \frac{u^g-v^g}{u - v} \cdot \prod_{\{u,v\} \in E(X)}\frac{(u^g)^h - (v^g)^h} {u^g-v^g}\\
&= \mathrm{sgn}_X(g) \ \mathrm{sgn}_X(h),
\end{align*}
where the second product in the penultimate line is equal to $\mathrm{sgn}_X(h)$ because $g \in \mathrm{Aut}(X)$.
\end{proof}

\begin{corollary}\label{lem:halfodd}
Exactly half of the permutations in $\mathrm{Aut}(X)$ are odd automorphisms for $X$ if and only if $X$ is an odd graph.
\end{corollary}

\begin{proof}
The following are equivalent: (a)~exactly half of the permutations in $\mathrm{Aut}(X)$ are odd automorphisms for $X$, 
(b) $\mathrm{sgn}_X$ maps onto $\{-1,1\}$, (c)~there exists $g\in\mathrm{Aut}(X)$ with $\mathrm{sgn}_X(g)=-1$, and (d)~$X$ is an odd graph.
\end{proof}

\begin{lemma}\label{le:cycle}
Suppose that $g \in \mathrm{Sym}(n)$. Then a cycle of $g_E$ contains an odd number of inversions of $g$ if and only if it is the undirected image of a self-paired cycle of $g_A$.
\end{lemma}

\begin{proof}
Let $C$ be a cycle of $g_A$ that is not self-paired, and let $C'$ be its undirected image. Label each arc $(u,v)$ of $C$ with $0$ if $u < v$ or $1$ otherwise, and count the number of times the label \emph{changes} (from $0$ to $1$ or vice-versa) as $C$ is traversed exactly once in cyclic order starting and ending at the same arc. This number is precisely the number of inversions of $g$ that lie in $C'$. As the label at the starting and ending point of this traversal is the same, the number of changes is even. Therefore $C'$ contains an \emph{even number} of inversions of $g$. 

Now suppose that $C'$ is the undirected image of a self-paired cycle 
\[
C = (a_1, a_2, \ldots, a_k, \overline{a_1},\overline{a_2},\ldots,\overline{a_k})
\]
of $g_A$. As previously, label each arc $(u,v)$ in $C$ with $0$ if $u<v$ and $1$ if $v < u$, and again count the total number of label changes as the cycle $C$ is traversed from $a_1$ to $\overline{a_1}$. (This corresponds to traversing $C'$ exactly once in cyclic order, but takes into account the fact that $g^k$ reverses the sense of $a_1$.)  As the start and end of this traversal have \emph{different} labels, there are an \emph{odd number} of label changes. Therefore the undirected image of any self-paired cycle of $g_A$ contains an \emph{odd number} of inversions of $g$.
\end{proof}

Returning to the example of $g=(1,2,3,4)$ from Section \ref{sec:graphcount}, we see that the first cycle of $g_E$ contains two inversions of $g$, namely $\{3,4\}$ and $\{1,4\}$, while the second cycle of $g_E$ contains just one.

Finally we have enough to count odd graphs.
\begin{theorem}\label{t:odd}
The number of isomorphism classes of odd graphs on $n$ vertices is given by the expression
\begin{equation}\label{eqn:oddgraph}
\frac{1}{n!} \sum_{\substack{g \in \mathrm{Sym}(n)\\ |g|\ \mathrm{even}}} 2^{c(g_E)}.
\end{equation}
\end{theorem}

\begin{proof}
We will double-count the elements of the following set
\[
S = \{(X, g) : g \text{ is an odd automorphism for the $n$-vertex odd graph } X\}.
\]

If $X$ is an odd graph, then there are $n!/|\mathrm{Aut}(X)|$ labelled graphs isomorphic to $X$, and by \cref{lem:halfodd}, each of these has $|\mathrm{Aut}(X)|/2$ odd automorphisms, meaning that each isomorphism class of odd graphs contributes $n!/2$ pairs to $S$. So if there are $K$ pairwise non-isomorphic odd graphs, then $|S| = K n!/2$.

Now, a permutation $g \in \mathrm{Sym}(n)$ is an odd automorphism for a graph $X$ if and only if $E(X)$ is a union of cycles of $g_E$ that (collectively) contain an \emph{odd number} of inversions of $g$.

If $g$ has odd order, then by \cref{le:selfpairedeven} every cycle of $g_A$ is non self-paired, and hence by \cref{le:cycle}, every cycle of $g_E$ contains an even number of inversions of $g$. Therefore $S$ contains no pairs $(X,g)$ for which $g$ has odd order.

If $g$ has even order, then $g_E$ has at least one cycle, say $C$, that is the undirected image of a self-paired cycle and therefore contains an odd number of inversions of $g$ (again using Lemmas \ref{le:selfpairedeven} and \ref{le:cycle}). 

There are $2^{c(g_E)-1}$ subsets of $E_n$ obtained by taking the union of a subset of the cycles of $g_E$ other than $C$. If $\widehat{E}$ is one of these subsets, then \emph{exactly one} of $\widehat{E}$ and $\widehat{E} \cup C$ contains an odd number of inversions of $g$. (If $\widehat{E}$ contains an even number of inversions of $g$, then $\widehat{E} \cup C$ contains an odd number of inversions of~$g$.)

Therefore when $g$ has even order, it contributes $2^{c(g_E)-1}$ pairs $(X,g)$ to $S$.

It follows that 
\[
|S| = K \frac{n!}{2} = \sum_{\substack{g \in \mathrm{Sym}(n)\\ |g|\ \mathrm{even}}} 2^{c(g_E)-1}
\]
and so $K$ is given by the expression~\eqref{eqn:oddgraph}.
\end{proof}

Every permutation of $\mathrm{Sym}(n)$ has even order or odd order, and so we obtain, from Theorems~\ref{t:tournaments} and~\ref{t:odd}, the final expression
\[
\underbrace{\frac{1}{n!} \sum_{g \in \mathrm{Sym}(n)} 2^{c(g_E)}}_{\text{\# Graphs}}
=
\underbrace{\frac{1}{n!} \sum_{\substack{g \in \mathrm{Sym}(n)\\ |g|\ \mathrm{odd}}} 2^{c(g_E)}}_{\text{\# Tournaments}}
+
\underbrace{\frac{1}{n!} \sum_{\substack{g \in \mathrm{Sym}(n)\\ |g|\ \mathrm{even}}} 2^{c(g_E)}}_{\text{\# Odd Graphs}}.
\]
Therefore, the number of isomorphism classes of $n$-vertex tournaments is equal to the number of isomorphism classes of $n$-vertex even graphs, and we have proved Theorem \ref{thm:evenandtourn}.

\backmatter

\bmhead{Acknowledgments}

The authors thank the Centre for the Mathematics of Symmetry and Computation at the University of Western Australia for supporting the 2022 CMSC Research Retreat where this problem was solved.

SDF was supported by a St Leonard's International Doctoral Fees Scholarship and a School of Mathematics \& Statistics PhD Funding Scholarship at the University of St Andrews.

SPG was supported by the Australian Research Council Discovery Project DP190100450.

\end{document}